\documentclass[12pt,reqno]{amsart}
\usepackage{mathrsfs}
\usepackage{amsfonts}
\usepackage{latexsym,amsmath,amsfonts,amssymb,amsthm,extarrows}
\theoremstyle{plain}
\newtheorem{Thm}{Theorem}
\newtheorem{Lem}{Lemma}

\newtheorem{Conj}{Conjecture}
\theoremstyle{definition}
\newtheorem*{Ack}{Acknowledgment}
\theoremstyle{remark}

\def\Q{\mathbb Q}

\def\1{{\bf 1}}

\def\pmod #1{\ ({\rm{mod}}\ #1)}

\def\qbinom #1#2#3{{\genfrac{[}{]}{0pt}{}{#1}{#2}}_{#3}}

\begin{document}
\title{On divisibility of sums of Ap\'ery polynomials}
\author{Hao Pan}
\begin{abstract} For any positive integers $m$ and $\alpha$, we prove that
$$\sum_{k=0}^{n-1}\epsilon^k(2k+1)A_k^{(\alpha)}(x)^m\equiv0\pmod{n},
$$
where $\epsilon\in\{1,-1\}$ and
$$
A_n^{(\alpha)}(x)=\sum_{k=0}^n\binom{n}{k}^{\alpha}\binom{n+k}{k}^{\alpha}x^k.$$
\end{abstract}
\address{Department of Mathematics, Nanjing University, Nanjing 210093,
People's Republic of China}\email{haopan79@yahoo.com.cn}
\thanks{}
\maketitle

\section{Introduction}
\setcounter{Lem}{0}\setcounter{Thm}{0}\setcounter{Cor}{0}
\setcounter{equation}{0}
The Ap\'ery number $A_n$ is defined by
$$
A_n=\sum_{k=0}^n\binom{n}{k}^2\binom{n+k}{k}^2.
$$
Those numbers play an important role in  Ap\'ery's ingredient proof \cite{Poorten78} of the irrationality of $\zeta(3)=\sum_{n=1}^\infty1/n^3$. In 2000, Ahlgren and Ono \cite{AhlgrenOno00} solved a conjecture of Beukers \cite{Beukers87} and showed that for odd prime $p$, $$A_{(p-1)/2}\equiv a(p)\pmod{p^2},$$ where $a(n)$ is the Fourier coefficient of $q^n$ in the modular form $\eta(2z)^4\eta(4z)^4$.

Recently, Sun \cite{Sun1} defined the Ap\'ery polynomial as
$$
A_n(x)=\sum_{k=0}^n\binom{n}{k}^2\binom{n+k}{k}^2x^k,$$
and proved several new congruences for the sums of $A_n(x)$.
For examples,
$$
\sum_{k=0}^{n-1}(2k+1)A_k(x)\equiv0\pmod{n}
$$
for every positive integer $n$. In fact, he showed that
$$
\frac1n\sum_{k=0}^{n-1}(2k+1)A_k(x)=\sum_{k=0}^{n-1}\binom{n-1}{k}\binom{n+k}k\binom{n+k}{2k+1}\binom{2k}kx^k.
$$
Furthermore, Sun proposed the following conjecture.
\begin{Conj} For $m\in\{1,2,3,\ldots\}$,
\begin{equation}
\sum_{k=0}^{n-1}\epsilon^k(2k+1)A_k(x)^m\equiv0\pmod{n},
\end{equation}
where $\epsilon\in\{1,-1\}$.
\end{Conj}
In \cite{GuoZeng}, Guo and Zeng proved that
$$
\frac1n\sum_{k=0}^{n-1}(-1)^k(2k+1)A_k(x)=(-1)^{n-1}\sum_{k=0}^{n-1}\binom{2k}kx^k\sum_{j=0}^k\binom{k}{j}\binom{k+j}j\binom{n-1}{k+j}\binom{n+k+j}{k+j}.
$$
On the other hand, in \cite{Sun}, Sun also define the central Delannoy polynomial
$$
D_n(x)=\sum_{k=0}^n\binom{n}{k}\binom{n+k}{k}x^k.
$$
He showed that
$$
\frac1n\sum_{k=0}^{n-1}(2k+1)D_k(x)=\sum_{k=0}^{n-1}\binom{n}{k+1}\binom{n+k}{k}x^k.
$$
Sun also conjectured that
$$
\frac1n\sum_{k=0}^{n-1}(2k+1)D_k(x)^m
$$
is always an integer.

In fact, motivated by \cite{Schmidt93} and \cite[eq. (1.7)]{AhlgrenOno00}, we may define the generalized Ap\'ery polynomial
$$
A_n^{(\alpha)}(x)=\sum_{k=0}^n\binom{n}{k}^{\alpha}\binom{n+k}{k}^{\alpha}x^k,$$
where $\alpha$ is a positive integer. (In \cite{GuoZeng}, Guo and Zeng called such polynomial as the Schmidt polynomial.) In the same paper, Guo and Zeng also proved that fact all $\alpha\geq 1$, there exist explicit formulas for
$$
\frac1n\sum_{k=0}^{n-1}(2k+1)A_k^{(\alpha)}(x)\text{\qquad and\qquad}\frac1n\sum_{k=0}^{n-1}(-1)^k(2k+1)A_k^{(\alpha)}(x).
$$
However, no explicit formula is known for
$$\frac1n\sum_{k=0}^{n-1}(2k+1)A_k^{(\alpha)}(x)^m\text{\qquad and\qquad}\frac1n\sum_{k=0}^{n-1}(-1)^k(2k+1)A_k^{(\alpha)}(x)^m
$$
when $m\geq 2$.

In this paper, we shall prove
\begin{Thm} For any positive integers $m$ and $\alpha$,
\begin{equation}\label{t1e1}
\sum_{k=0}^{n-1}(2k+1)A_k^{(\alpha)}(x)^m\equiv0\pmod{n},
\end{equation}
and
\begin{equation}\label{t1e2}
\sum_{k=0}^{n-1}(-1)^k(2k+1)A_k^{(\alpha)}(x)^m\equiv0\pmod{n}.
\end{equation}
\end{Thm}
In the next sections, we shall use $q$-congruences to prove (\ref{t1e1}) and (\ref{t1e2}) respectively.

\section{Proof of (\ref{t1e1})}
\setcounter{Lem}{0}\setcounter{Thm}{0}\setcounter{Cor}{0}
\setcounter{equation}{0}
For an integer $n$, define the $q$-integer
$$
[n]_q=\frac{1-q^n}{1-q}.
$$
Clearly $\lim_{q\to 1}[n]_q=n$.
For a non-negative integer $k$, the $q$-binomial coefficient $\qbinom{n}{k}q$ is given by
$$
\qbinom{n}{k}q=\frac{\prod_{1\leq j\leq k}[n-j+1]_q}{\prod_{1\leq j\leq k}[j]_q}.
$$
In particular, $\qbinom{n}{0}q=1$. Also, we set $\qbinom{n}{k}q=0$ if $k<0$. It is easy to see that$\qbinom{n}{k}q$ is a polynomial in $q$ with integral coefficients, since
$$
\qbinom{n+1}{k}q=q^k\qbinom{n}{k}q+\qbinom{n}{k-1}q.
$$

Below we introduce the notion of $q$-congruences. Suppose that $a,b,n$ are integers and $a\equiv b\pmod{n}$.  Then over the polynomial ring $\Q(q)$, we have
$$
\frac{1-q^a}{1-q}-\frac{1-q^b}{1-q}=q^a\cdot\frac{1-q^{b-a}}{1-q}\equiv 0\pmod{\frac{1-q^n}{1-q}},
$$
i.e., $[a]_q\equiv [b]_q\pmod{[n]_q}$. Furthermore, for the $q$-binomial coefficients, we have the following  $q$-Lucas congruence.
\begin{Lem}\label{qLucas} Suppose that $d>1$ is a positive integer. Suppose that $a,b,h,l$ are integers with $0\leq b,l\leq  d-1$. Then
$$
\qbinom{ad+b}{hd+l}q\equiv\binom{a}{h}\qbinom{b}{l}q\pmod{\Phi_d(q)},
$$
where $\Phi_d(q)$ is the $d$-th cyclotomic polynomial.
\end{Lem}
Define the generalized $q$-Ap\'ery polynomial
$$
A_k^{(\alpha)}(x;q)=\sum_{j=0}^{k}q^{\binom{j}{2}-jk}\qbinom{k}{j}{q}^\alpha\qbinom{k+j}{j}{q}^\alpha x^j.
$$
In order to prove (\ref{t1e1}), it suffices to show that
\begin{Thm}\label{qt1}
\begin{equation}\label{qt1e}\sum_{k=0}^{n-1}q^{n-1-k}[2k+1]_qA_k^{(\alpha)}(x;q)\equiv0\pmod{[n]_q}.
\end{equation}
\end{Thm}
Let us explain why (\ref{qt1e}) implies (\ref{t1e1})). Since $[n]_q$ is a primitive polynomial (i.e., the greatest divisor of all coefficients of $[n]_q$ is 1), by the Gauss lemma, there exists a polynomial $H(x,q)$ with integral coefficients such that
\begin{equation}\label{qconH}
\sum_{k=0}^{n-1}q^{n-1-k}[2k+1]_qA_k^{(\alpha)}(x;q)=[n]_qH(x,q).
\end{equation}
Substituting $q=1$ in (\ref{qconH}), we get (\ref{t1e1}).

It is not difficult to check that
$$
[n]_q=\prod_{\substack{d\mid n\\ d>1}}\Phi_d(q).
$$
The advantage of $q$-congruences is that we only need to prove that
$$
\sum_{k=0}^{n-1}q^{n-1-k}[2k+1]_qA_k^{(\alpha)}(x;q)\equiv0\pmod{\Phi_d(q)}
$$
for every divisor $d>1$ of $n$. Note that
\begin{align*}
\qbinom{k+j}{j}{q}=&\frac{(1-q^{k+1})(1-q^{k+2})\cdots(1-q^{k+j})}{(1-q)(1-q^2)\cdots(1-q^j)}\\=&(-1)^j\frac{q^{jk+\binom{j+1}2}(1-q^{-k-1})(1-q^{-k-2})\cdots(1-q^{-k-j})}{(1-q)(1-q^2)\cdots(1-q^j)}\\
=&(-1)^jq^{jk+\binom{j+1}2}\qbinom{-k-1}{j}q.
\end{align*}
So
$$
A_k^{(\alpha)}(x;q)=\sum_{-\infty<j<+\infty}(-1)^{\alpha j}q^{\alpha j^2}\qbinom{k}{j}{q}^\alpha\qbinom{-k-1}{j}{q}^\alpha x^j.
$$
Suppose that $d>1$ is a divisor of $n$. Let $h=n/d$. Write $k=ad+b$ where $0\leq b\leq d-1$. Then by Lemma \ref{qLucas},
\begin{align*}
&A_{ad+b}^{(\alpha)}(x;q)\\
=&\sum_{\substack{-\infty<s<+\infty\\ 0\leq t\leq d-1}}(-1)^{\alpha(sd+t)}q^{\alpha (sd+t)^2}\qbinom{ad+b}{sd+t}{q}^\alpha\qbinom{(-a-1)d+d-b-1}{sd+t}{q}^\alpha x^{sd+t}\\
\equiv&\sum_{\substack{-\infty<s<+\infty\\ 0\leq t\leq d-1}}(-1)^{\alpha(sd+t)}q^{\alpha t^2}\binom{a}{s}^\alpha\qbinom{b}{t}{q}^\alpha\binom{-a-1}{s}^\alpha\qbinom{d-b-1}{t}{q}^\alpha x^{sd+t}\pmod{\Phi_d(q)}.
\end{align*}
Hence,
\begin{align*}\sum_{k=0}^{n-1}[2k+1]_qq^{n-1-k}A_k^{(\alpha)}(x;q)^m
=&\sum_{\substack{0\leq a\leq h-1\\ 0\leq b\leq d-1}}q^{hd-1-ad-b}[2ad+2b+1]_qA_{ad+b}^{(\alpha)}(x;q)^m\\
\equiv&\sum_{\substack{0\leq a\leq h-1\\ 0\leq b\leq d-1}}q^{-1-b}[2b+1]_qB_{a,b,d}^{(\alpha)}(x;q)^m\pmod{\Phi_d(q)},
\end{align*}
where
$$
B_{a,b,d}^{(\alpha)}(x;q)=\sum_{\substack{-\infty<s<+\infty\\ 0\leq t\leq d-1}}(-1)^{\alpha(sd+t)}q^{\alpha t^2}\binom{a}{s}^\alpha\qbinom{b}{t}{q}^\alpha\binom{-a-1}{s}^\alpha\qbinom{d-b-1}{t}{q}^\alpha x^{sd+t}.
$$
Similarly, since $k=ad+b\Longleftrightarrow n-k-1=(h-a-1)d+(d-b-1)$ and $B_{a,b,d}^{(\alpha)}(x;q)=B_{a,d-b-1,d}^{(\alpha)}(x;q)$, we have
\begin{align*}\sum_{k=0}^{n-1}q^k[2n-2k-1]_qA_{n-k-1}^{(\alpha)}(x;q)^m\equiv&\sum_{\substack{0\leq a\leq h-1\\ 0\leq b\leq d-1}}q^{b}[-2b-1]_qB_{h-a-1,d-b-1,d}^{(\alpha)}(x;q)^m\\
\xlongequal{a'=h-a-1}&\sum_{\substack{0\leq a'\leq h-1\\ 0\leq b\leq d-1}}q^{b}[-2b-1]_qB_{a',b,d}^{(\alpha)}(x;q)^m\pmod{\Phi_d(q)}.\end{align*}
Note that
\begin{align*}
q^{-1-b}[2b+1]_q+q^{b}[-2b-1]_q=q^{-1-b}-q^{b}+q^{b}-q^{-b-1}=0.
\end{align*}
Therefore,
\begin{align*}&2\sum_{k=0}^{n-1}q^{n-1-k}[2k+1]_qA_k^{(\alpha)}(x;q)^m\\
=&\sum_{k=0}^{n-1}q^{n-1-k}[2k+1]_qA_k^{(\alpha)}(x;q)^m+\sum_{k=0}^{n-1}q^{k}[2n-2k-1]_qA_{n-1-k}^{(\alpha)}(x;q)^m\\
\equiv&0\pmod{\Phi_d(q)}.
\end{align*}
This concludes the proof of Theorem \ref{qt1}.

\section{Proof of (\ref{t1e2})}
\setcounter{Lem}{0}\setcounter{Thm}{0}\setcounter{Cor}{0}
\setcounter{equation}{0}
The proof of (\ref{t1e2}) is a little complicated.
\begin{Thm}
$$\sum_{k=0}^{n-1}(-1)^kq^{n-1-k}[2k+1]_qA_k^{(\alpha)}(x;q^2)$$
is divisible by
$$\prod_{\substack{d\mid n\\ d>1\text{ is odd}}}\Phi_d(q)\cdot\prod_{\substack{d\mid n\\ d\text{ is even}}}\Phi_d(q^2).
$$
\end{Thm}
Clearly we only need to prove that
$$
\sum_{k=0}^{n-1}(-1)^kq^{n-1-k}[2k+1]_qA_k^{(\alpha)}(x;q^2)^m+\sum_{k=0}^{n-1}(-1)^{n-1-k}q^{k}[2n-2k-1]_qA_{n-1-k}^{(\alpha)}(x;q^2)^m.
$$
is divisible by $\Phi_d(q)$ for odd $d>1$ and by $\Phi_d(q^2)$ for even $d$ respectively.
\begin{Lem} If $d>1$ is odd, then $\Phi_d(q)$ divides $\Phi_{d}(q^2)$. If $d$ is even, then $\Phi_d(q^2)=\Phi_{2d}(q)$.
\end{Lem}
\begin{proof}
We know that for $d>1$,
$$
\Phi_d(q)=\prod_{\xi\text{ is }d\text{-th primitive root of unity}}(q-\xi).
$$
Suppose that $d$ is odd and $\xi$ is an arbitrary $d$-th primitive root of unity. Then $\xi^2$ also is a $d$-th primitive root of unity, i.e., $\Phi_d(\xi^2)=0$. Hence $\Phi_d(q)$ divides $\Phi_{d}(q^2)$. Similarly, if $d$ is even and $\xi$ is a $2d$-th primitive root of unity, then $\xi^2$ is a $d$-th primitive root of unity. So $\Phi_{2d}(q)$ divides $\Phi_d(q^2)$. Note that now $\deg\Phi_{2d}=\phi(2d)=2\phi(d)=2\deg\Phi_d$, where $\phi$ is the Euler totient function. We must have $\Phi_d(q^2)=\Phi_{2d}(q)$.
\end{proof}

Suppose that $d>1$ is an odd divisor of $n$. Let $h=n/d$. Then
\begin{align*}
&\sum_{k=0}^{n-1}(-1)^{k}q^{n-1-k}[2k+1]_qA_{k}^{(\alpha)}(x;q^2)^m\\
\equiv&\sum_{\substack{0\leq a\leq h-1\\ 0\leq b\leq d-1}}(-1)^{ad+b}q^{hd-1-ad-b}[2(ad+b)+1]_qB_{a,b,d}^{(\alpha)}(x;q^2)^m\pmod{\Phi_d(q^2)}\\
\equiv&\sum_{\substack{0\leq a\leq h-1\\ 0\leq b\leq d-1}}(-1)^{ad+b}q^{-1-b}[2b+1]_qB_{a,b,d}^{(\alpha)}(x;q^2)^m\pmod{\Phi_d(q)}.
\end{align*}
and
\begin{align*}
&\sum_{k=0}^{n-1}(-1)^{n-1-k}[2n-2k-1]_qq^{k}A_{n-1-k}^{(\alpha)}(x;q^2)^m\\
\equiv&\sum_{\substack{0\leq a\leq h-1\\ 0\leq b\leq d-1}}(-1)^{hd-1-ad-b}q^{ad+b}[2sd-2(ad+b)-1]_qB_{h-a-1,d-b-1,d}^{(\alpha)}(x;q^2)^m\\
\equiv&\sum_{\substack{0\leq a'\leq h-1\\ 0\leq b\leq d-1}}(-1)^{a'd+d-1-b}[-2b-1]_qq^{b}B_{a',b,d}^{(\alpha)}(x;q^2)^m\pmod{\Phi_d(q)}.
\end{align*}
Since $d$ is odd,
$$
(-1)^{ad+b}q^{-1-b}[2b+1]_q+(-1)^{ad+d-1-b}q^{b}[-2b-1]_q=0.$$
So
$\Phi_d(q)$ divides
$$
\sum_{k=0}^{n-1}(-1)^kq^{n-1-k}[2k+1]_qA_k^{(\alpha)}(x;q^2)^m+\sum_{k=0}^{n-1}(-1)^{n-1-k}q^{k}[2n-2k-1]_qA_{n-1-k}^{(\alpha)}(x;q^2)^m.
$$

Suppose that $d$ is an even divisor of $n$.
Then\begin{align*}&\sum_{k=0}^{n-1}(-1)^kq^{n-1-k}[2k+1]_qA_k^{(\alpha)}(x;q^2)^m\\
\equiv&\sum_{\substack{0\leq a\leq h-1\\ 0\leq b\leq d-1}}(-1)^{ad+b}q^{hd-1-(ad+b)}[2(ad+b)+1]_qB_{a,b,d}^{(\alpha)}(x;q^2)^m\\
\equiv&\sum_{\substack{0\leq a\leq h-1\\ 0\leq b\leq d-1}}(-1)^{ad+b}q^{hd-ad-1-b}[2b+1]_qB_{a,b,d}^{(\alpha)}(x;q^2)^m\pmod{\Phi_d(q^2)}.
\end{align*}
And
\begin{align*}&\sum_{k=0}^{n-1}(-1)^{n-1-k}q^{k}[2n-2k-1]_qA_{n-1-k}^{(\alpha)}(x;q^2)^m\\
\equiv&\sum_{\substack{0\leq a\leq h-1\\ 0\leq b\leq c-1}}(-1)^{hd-1-(ad+b)}q^{ad+b}[2hd-2(ad+b)-1]_qB_{h-a-1,d-b-1,d}^{(\alpha)}(x;q^2)^m\\
\equiv&\sum_{\substack{0\leq a'\leq h-1\\ 0\leq b\leq d-1}}(-1)^{a'd+d-1-b}q^{hd-a'd-d+b}[-2b-1]_qB_{a',b,d}^{(\alpha)}(x;q^2)^m\pmod{\Phi_d(q^2)}.
\end{align*}
Note that
$\Phi_d(q^2)=\Phi_{2d}(q)$ divides $1+q^d=(1-q^{2d})/(1-q^d)$, i.e., $$q^d\equiv -1\pmod{\Phi_d(q^2)}.$$ We have
\begin{align*}
&(-1)^{ad+b}q^{hd-ad-1-b}[2b+1]_q+(-1)^{ad+d-1-b}q^{hd-ad-d+b}[-2b-1]_q\\
\equiv&(-1)^{ad+b}q^{hd-ad}(q^{-1-b}[2b+1]_q+q^{b}[-2b-1]_q)\\
=&0\pmod{\Phi_d(q^2)}.
\end{align*}
That is, $\Phi_d(q^2)$ divides
$$
\sum_{k=0}^{n-1}(-1)^kq^{n-1-k}[2k+1]_qA_k^{(\alpha)}(x;q^2)^m+\sum_{k=0}^{n-1}(-1)^{n-1-k}q^{k}[2n-2k-1]_qA_{n-1-k}^{(\alpha)}(x;q^2)^m.
$$
\begin{Ack} I am grateful to Professor Zhi-Wei Sun for his helpful suggestions on this paper.
\end{Ack}

\end{document}